\documentclass[12pt,reqno]{article}
\usepackage{amssymb}
\usepackage{amsmath}
\usepackage{amsthm}
\usepackage{amsfonts}
\usepackage{url}
\usepackage{hyperref}
\usepackage{breakurl}
\usepackage[T1]{fontenc}	

\newcommand{\R}{{\mathbb R}}
\newcommand{\N}{{\mathbb N}}

\newcommand{\Z}{{\mathbb Z}}

\newcommand{\raisecomma}{\raisebox{2pt}{$,$}}
\newcommand{\raisedot}{\raisebox{2pt}{$.$}}
\newcommand{\squash}{\vspace*{-10pt}}

\newtheorem{theorem}{Theorem}
\newtheorem{corollary}{Corollary}

\newtheorem{proposition}{Proposition}
\newtheorem{definition}{Definition}
\newtheorem{remark}{Remark}
\newtheorem{conjecture}{Conjecture}

\newcounter{spectrum}
\newcounter{maxdet}
\newcounter{depth}
\newcounter{bounds_15_16}
\newcounter{freqs_15_7}
\newcounter{order1}
\newcounter{order12}
\newcounter{order16a}
\newcounter{order16d}
\newcounter{order18a}
\newcounter{order18b}
\newcounter{order18c}
\newcounter{order21e}

\newcommand{\seqnum}[1]{\href{http://oeis.org/#1}{\underline{#1}}}

\begin{document}
\bibliographystyle{plain}
\title{On minors of maximal determinant matrices}
\author{Richard P. Brent\\
Mathematical Sciences Institute\\
Australian National University\\
Canberra, ACT 0200,
Australia\\
\href{mailto:minors@rpbrent.com}{\tt minors@rpbrent.com}
\and
Judy-anne H. Osborn\\
CARMA\\
The University of Newcastle\\
Callaghan, NSW 2308,
Australia\\
\href{mailto:Judy-anne.Osborn@newcastle.edu.au}%
{\tt Judy-anne.Osborn@newcastle.edu.au}
}

\date{}	

\maketitle
\thispagestyle{empty}                   

\begin{abstract}
By an old result of Cohn (1965),
a Hadamard matrix of order $n$ has no proper Hada\-mard submatrices of order
$m > n/2$. We generalise this result to maximal determinant submatrices of
Hadamard matrices, and show that an interval of length $\sim n/2$ is
excluded from the allowable orders. We make a conjecture regarding a lower
bound for sums of squares of minors of maximal determinant matrices, and
give evidence to support it. We give tables of the
values taken by the minors of all maximal determinant matrices of orders
$\le 21$ and make some observations on the data. Finally, we describe the
algorithms that were used to compute the tables.
\end{abstract}

\pagebreak[3]
\section{Introduction}		\label{sec:intro}

A $\{+1, -1\}$-matrix (abbreviated ``$\{\pm 1\}$-matrix'' below) is
a matrix whose elements are $+1$ or $-1$.
The \emph{Hadamard maximal determinant problem}, posed by
Hadamard~\cite{Hadamard}, is to find the maximal determinant $D(n)$
of an $n \times n$ $\{\pm 1\}$-matrix of given order $n$.
A matrix that attains the maximum is a \emph{maximal determinant matrix}
(abbreviated \emph{maxdet matrix}).
Such matrices, of which Hadamard
matrices are a special case, are of interest in combinatorics and have
applications in statistical design~\cite{KhOr06}, coding theory and signal
processing~\cite{Horadam07,Orrick-www,SeYa92}.  In design theory
a maxdet matrix is also known as a 
\emph{saturated D-optimal design}. 

\pagebreak[3]
Finding maxdet matrices is in general difficult, with the
difficulty depending on the congruence class of $n \pmod 4$.
Most is known for $n \equiv 0 \pmod 4$ (the \emph{Hadamard orders}), and
least for $n \equiv 3 \pmod 4$.\footnote{See 
{\DJ}okovic and Kotsireas~\cite{DK} for a recent
summary of what is known about the cases $n \equiv 2 \pmod 4$,
and Brent \emph{et al}~\cite{rpb244} for the cases $n \equiv 1 \pmod 2$.}
  Maxdet matrices are known for
orders $1$ through $21$ inclusive; for $n=22$
and $n=23$ there are conjectures (see~\cite{Orrick-www}) but as yet
no proofs of maximality.
Constructions exist for various infinite families and
ad-hoc examples, see~\cite{Osborn02}.

Consider the $k\times k$ submatrices ${S}_k(A)$ 
of an $n \times n$ $\{\pm 1\}$-matrix $A$,
where $0 < k \le n$.
For $M \in {S}_k(A)$, we say that $\det(M)$ is
a \emph{minor of order $k$} of~$A$. Usually only the absolute
value of $\det(M)$ is of interest.

One method for finding maxdet or large-determinant matrices involves
choosing a large-determinant matrix of a slightly smaller (or larger) order 
than the desired order, possibly perturbing it by a low-rank modification,
and adding (or removing) a small number of suitably chosen rows and columns.  
For example, Solomon~\cite{Orrick-www,Solomon02} found a (conjectured)
maxdet matrix of order $33$ and determinant $441 \times 2^{74}$
in this way by starting with an appropriate
Hadamard matrix of order $32$,
and in~\cite{rpb249,KMS00,LL} the method was used 
to obtain lower bounds on $D(n)$.

This motivates our interest in the
minors of maxdet matrices, and in particular the question:
\emph{What is the largest order of a maxdet matrix
contained as a proper submatrix of a given maxdet matrix?}
In this paper we answer this question for the maxdet
submatrices of maxdet matrices of orders $n \le 21$ by
computing all minors of maxdet matrices of these orders.  
Our work extends that of earlier researchers 
who have considered minors of Hadamard and maxdet matrices with other
applications in mind, such as the problem of growth in Gaussian
elimination~\cite{DP88,KLMS03,KMS01,KM08,SXKM03}.

Schmidt~\cite[pg.~441]{Schmidt70} says \emph{``The nature of the construction
$\cdots$ is in line with the computer assisted observation that binary
matrices with maximal determinants may not contain large order submatrices
with large determinants''}. Whether this is true in general depends on the
precise meaning of ``large''.  Certainly there are exceptions.
For example, the maxdet $\{\pm 1\}$-matrix 
of order~$17$ contains a maxdet (Hadamard) submatrix of order~$16$.

In Theorem~\ref{thm:excluded_interval} of \S\ref{sec:Sz} 
we give a new proof of an old result of Cohn~\cite{Cohn65} 
that a Hadamard matrix of order $n$ can not have
a proper Hadamard submatrix of order $m > n/2$. We then generalize the
proof to cover maxdet
submatrices of Hadamard matrices, and show, assuming the
Hadamard conjecture, that a Hadamard matrix of order $n$ can not have a
maxdet submatrix $M$ of order $m \ge n/2 + 5\log n$ unless $m \ge
n-2$ (see Theorem~\ref{thm:excluded_interval2}). 
In other words, $m$ can not lie in the interval
$[n/2+5\log n, n-2)$. Without the assumption of the Hadamard conjecture, 
we can still exclude an interval $[n/2 + o(n), n-o(n))$ of length $\sim
n/2$, see Theorem~\ref{thm:unconditional} and Remark~\ref{remark:BHP}.
These results partially confirm the remark of Schmidt quoted above.
However, they apply only to sub-matrices of Hadamard matrices. 
Except in small cases that are amenable to explicit computation, 
we have not excluded 
the possibility that a maxdet matrix of order 
$n \not\equiv 0 \pmod 4$ 
has a maxdet submatrix of any given order $m < n$.

In \S\ref{sec:numerics} we define two sequences 
related to the sets of minors of maxdet matrices,
and give the first $21$ terms of each sequence.

In \S\ref{sec:data} we describe the minors that occur in maxdet
matrices of orders $1$ through $21$, and make some observations
on the patterns that occur. We mention a result
(Proposition~\ref{prop:small_minors}), on small minors of Hadamard
matrices, which was suggested by the data before a proof was found.

Motivated by Tur\'an's result~\cite{Turan40} 
that the expected value of $\det(A)^2$
is $m!$ for random $\{\pm 1\}$ matrices of order $m$,
we consider the mean value of $\det(M)^2$ over all $m \times m$ submatrices
of maxdet matrices of order $n \ge m$, and conjecture
that it is bounded below by $m!$ (see Conjecture~\ref{conj:conj_lowerbd}).
The conjecture is consistent with our computations 
for $1 \le m \le n \le 21$.

Finally, 
in \S\ref{sec:alg} we describe the algorithms that we used to compute minors
of square $\{\pm 1\}$-matrices, as well as some related algorithms that were
considered but rejected for various reasons.

\pagebreak[3]

\subsubsection*{Hadamard equivalence and HT-equivalence%
\footnote{There seems to be no
widely-accepted name for this concept. 
In~\cite{rpb245} HT-equivalence is called ``extended Hadamard equivalence''.
Wanless~\cite{Wanless05}
calls an HT-equivalence class a ``resemblance class''.}
}

Two $\{\pm 1\}$-matrices 
are \emph{Hadamard equivalent} if one can be obtained from
the other by negating rows or columns, and/or by interchanging rows or columns.
Clearly the answers to the questions posed above about minors 
are the same for all matrices in a
Hadamard equivalence class, and also for any matrix $A$ and its transpose
(dual) $A^T$.  Hence, it is useful to define the notion of 
\emph{HT-equivalence}
by saying that two matrices $A$ and $B$ are \emph{HT-equivalent}
if $A$ is Hadamard-equivalent to $B$ \emph{or}
$A$ is Hadamard-equivalent to $B^T$.
For example, there are~$5$ distinct
Hadamard equivalence classes of Hadamard matrices of order~$16$, but two of
these classes contain matrices that are dual to those in the other class, so there are 
only $4$ distinct HT-equivalence classes.

\pagebreak[3]

\section{Excluded minors of Hadamard matrices} \label{sec:Sz}

In this section we consider the possible orders of submatrices $M$
of a Hadamard matrix $H$, satisfying the condition that $M$ is
Hadamard (see Theorem~\ref{thm:excluded_interval}) or, more generally, 
that $M$ is a maxdet matrix (see
Theorems~\ref{thm:excluded_interval2}--\ref{thm:unconditional}).

Recently Sz\"oll\H{o}si~\cite[Proposition 5.5]{Szollosi10} established
an elegant correspondence between the minors of order $m$ and of order $n-m$
of a Hadamard matrix of order $n$.  His result applies to complex Hadamard
matrices, of which $\{\pm 1\}$ Hadamard matrices are a special case.
More precisely, if $d+m = n$, $0 < d < n$, then 
corresponding to each $m\times m$ submatrix with determinant $\mu$ there is
a $d\times d$ submatrix with determinant $\pm n^{n/2 - d}\mu$. 
Only special cases (for small~$d$ or $m$) were known before
Sz\"oll\H{o}si
(see for example~\cite{DP88,KMS01,SXKM03,Sharpe07}).
This is perhaps surprising as 
Sz\"oll\H{o}si's critical Lemma~5.7 follows in a straight\-forward manner 
from Jacobi's determinant identity~\cite{Jacobi}.
We use the following corollary of Sz\"oll\H{o}si's theorem.%
\footnote{Corollary \ref{cor:sz} is essentially the same as Theorem~3 of
Cohn~\cite{Cohn65}, the difference being that Cohn replaces $D(n-m)$
by the Hadamard bound.  
However, Cohn's proof is quite different from ours.}
\begin{corollary} \label{cor:sz}
Suppose that a Hadamard matrix $H$ of order $n$ has an $m\times m$ submatrix
$M$, where $n/2 \le m \le n$.  Then
\[|\det(M)| \le n^{m-n/2}D(n-m)\,.\]
\end{corollary}
\begin{proof}
By Sz\"oll\H{o}si's theorem,
$|\det(M)| = n^{n/2-d}|\det(M')|$, where $d = n-m$ and
$M'$ is some $d\times d$ submatrix of $H$.
Since $n/2-d = m-n/2$ and $|\det(M')| \le D(d)$ by the definition of $D$,
the corollary follows.
\end{proof}
The following lower bound on $D(m)$ is given 
in~\cite[Corollary~3]{rpb249}. 
\begin{proposition} \label{prop:lowerbd1}
Assume the Hadamard conjecture.\footnote{It is sufficient to 
assume that Hadamard matrices of order $4k$ exist for all positive integers
$k \le (m+2)/4$. This is known to be true for $4k < 668$,
see~\cite{KT,SeYa92}.}
Then, for $m \ge 4$, we have $D(m) \ge 4m^{m/2-1}$.
\end{proposition}
Suppose that a Hadamard matrix $H$ of order $n$ has a maxdet
submatrix $M$ of order $m$.
Corollary~\ref{cor:sz} gives an upper bound on $|\det(M)|$,
and Proposition~\ref{prop:lowerbd1} gives a lower bound.
Theorems~\ref{thm:excluded_interval}--\ref{thm:excluded_interval2}
show that these bounds are incompatible for certain values of~$m$.
Theorem~\ref{thm:excluded_interval} considers the case that $M$ is
a Hadamard matrix, and Theorem~\ref{thm:excluded_interval2} considers
the more general case that $M$ is a maxdet matrix.
We have stated Theorem~\ref{thm:excluded_interval2} under the assumption
of the Hadamard conjecture, but a weaker result is provable without
this assumption~-- see Theorem~\ref{thm:unconditional}.

Theorem~\ref{thm:excluded_interval} was proved
by Cohn~\cite[Theorem~2]{Cohn65}, but we give a different proof
which generalizes to give proofs of 
Theorems~\ref{thm:excluded_interval2}--\ref{thm:unconditional}.

\pagebreak[3]
\begin{theorem}[Cohn] \label{thm:excluded_interval}
Let $H$ be a Hadamard matrix of order $n$ having a Hadamard
submatrix $M$ of order $m < n$. Then $m \le n/2$.
\end{theorem}
\begin{proof}
The theorem is trivial if $m \le n/2$, 
so assume that $n > m > n/2$. Let $d = n-m$.
Since $M$ exists and $|\det(M)| = m^{m/2}$, 
Corollary~\ref{cor:sz} and Hadamard's bound for $D(d)$ give
\begin{equation} \label{eq:ineq2}
m^{m/2} \le n^{m-n/2}d^{d/2}\,.
\end{equation}
Squaring both sides of~\eqref{eq:ineq2},
we have $(m/n)^m \le (d/n)^d$. Taking $n$-th roots and defining
$x := m/n \in (0,1)$, we see that \[x^x \le (1-x)^{1-x}.\]
This inequality
is equivalent to $f(x) \le 0$, where $f:[0,1]\to\R$ is defined by
\[
f(x) = \begin{cases} x\ln x - (1-x)\ln(1-x) \;\text{ if }\; x \in (0,1),\\
	0 \text{ otherwise.}
        \end{cases} 
\]
It is easy to verify that $f(1/2) = 0$,
$f'(x) = 2 + \ln x + \ln(1-x)$,
and
\[
f''(x) = \frac{1-2x}{x(1-x)} < 0 \;\text{ in } (1/2, 1).
\]
Thus, $f(x) > 0$ in $(1/2,1)$, so we must have $x \le 1/2$ or $x = 1$.
The case $x=1$ is ruled out because it implies that $m=n$, contrary
to the assumption that $m<n$.
Thus $x \le 1/2$, which implies that $m \le n/2$.
\end{proof}
\begin{theorem} \label{thm:excluded_interval2}
Assume the Hadamard conjecture.
Let $H$ be a Hadamard matrix of order $n$ having a maxdet
submatrix $M$ of order $m < n$. Then 
$m < (n/2 + 5\ln n)$ or $m \ge n - 2$.
\end{theorem}
\begin{proof}
The result is trivial if $m \le n/2$, and
$n/2 + 5\ln n > n-3$ for $n \le 28$, so we can assume that $m > n/2 > 14$. 
By Proposition~\ref{prop:lowerbd1} and 
Corollary~\ref{cor:sz}, we have
\begin{equation} 	\label{eq:ineq_stronger}
m^{m/2}\left(\frac{4}{m}\right) \le n^{m-n/2}d^{d/2},
\end{equation}
where $d = n-m$.
We prefer to use the slightly weaker inequality
\begin{equation}	\label{eq:ineq_weaker}
m^{m/2}\left(\frac{4}{n}\right) \le n^{m-n/2}d^{d/2}.
\end{equation}
Taking logarithms, and defining $x := m/n$ and 
$f$ as in the proof of Theorem~\ref{thm:excluded_interval},
we obtain
\begin{equation}	\label{eq:ineq_fn}
f(x) \le \frac{2\ln(n/4)}{n}\raisedot
\end{equation}
The right side of this inequality is positive (since $n > 4$) and 
independent of~$x$ (this is why we used~\eqref{eq:ineq_weaker} instead of
\eqref{eq:ineq_stronger}). We showed in the proof of 
Theorem~\ref{thm:excluded_interval} that $f(1/2) = f(1) = 0$
and $f''(x) < 0$ on $(1/2,1)$.
Let \[x_{\rm max} = (1 + \sqrt{1-4/e^2})/2 \approx 0.84\] be the 
(unique) point in $(1/2,1)$ at which $f'(x)$ vanishes.
Thus $f(x)$ attains its maximum value
at $x = x_{\rm max}$.
Since $2\ln(n/4)/n < 0.14 < f(x_{\rm max}) \approx 0.15$ for $n > 28$, the 
inequality~\eqref{eq:ineq_fn} is not satisfied for all $x \in (1/2,1)$,
and there is a unique interval $(x_0,x_1) \subseteq (1/2,1)$ on which
$f(x) > 2\ln(n/4)/n$, with $x_{\rm max} \in (x_0, x_1)$. 
(Here $x_0$ and $x_1$ depend on $n$ but not on $m$.)
It follows that there can not exist 
a maxdet submatrix of order $m$ with $x_0 < m/n < x_1$.

To locate $x_1$ we consider the case $d = n - m = 3$.
The inequality~\eqref{eq:ineq_weaker} gives
\begin{equation}	\label{eq:ineq28}
\left(\frac{n-3}{n}\right)^{n-3} \le \frac{27}{16n} \,\raisecomma
\end{equation}
but the left side of this inequality is bounded away from zero 
as $n \to \infty$,
whereas the right side tends to zero.  Thus, the inequality can not hold
for large $n$. In fact, a computation shows that
\eqref{eq:ineq28} can only hold for $n < 29$.
Thus, for $n \ge 29$ an interval $(x_0,x_1)$ as above exists,
with $x_1 > 1- 3/n$, so $m=n-3$ is not a possible order of a maxdet 
submatrix $M$.

We now show that $nx_0 < n/2 + 5\ln n$.
Define $\nu := n/2$,
$\delta := m - \nu$.  Thus $m = \nu + \delta$, 
$d = \nu - \delta$, 
and squaring the inequality~\eqref{eq:ineq_weaker} gives
\[(\nu+\delta)^{\nu+\delta}\left(\frac{2}{\nu}\right)^2
 \le (2\nu)^{2\delta} (\nu - \delta)^{\nu - \delta}\,,\]
or equivalently
\begin{equation}  \label{eq:ineq2a}
\left(\frac{\nu+\delta}{\nu-\delta}\right)^\nu \le
 \left(\frac{\nu}{2}\right)^2
 \left(\frac{4\nu^2}{\nu^2-\delta^2}\right)^\delta\,.
\end{equation}
Write $z := \delta/\nu = 2x-1$. We can assume that 
\[0 < z < z_{max} = 2x_{max}-1 = \sqrt{1-4/e^2} \approx 0.68\,.\]
Taking logarithms in~\eqref{eq:ineq2a} gives
\begin{equation} \label{eq:ineq2b}
\frac{\delta}{z}\ln\left(\frac{1+z}{1-z}\right) \le 
 \delta(\ln 4 - \ln(1-z^2)) +
 2\ln(\nu/2).
\end{equation}
Collecting the terms involving $\delta$ gives
\[
\delta(2-\ln4 - \varepsilon(z)) \le 2\ln(\nu/2),
\]
where 
\[
\varepsilon(z) = 2 - \frac{1}{z}\ln\left(\frac{1+z}{1-z}\right) - 
 \ln(1-z^2) = \sum_{k=1}^\infty \frac{z^{2k}}{k(2k+1)} 
\]
is monotonic increasing on $[0,z_{\rm max}]$, so
\[\varepsilon(z) \le \varepsilon(z_{\rm max}) < 0.1803,\]
and thus $2 - \ln 4 - \varepsilon(z) > 0.4334$ for $z \in (0, z_{\rm max})$.
It follows that
\[
\delta < \frac{2\ln(\nu/2)}{0.4334} < 5\ln n.
\]
\end{proof}

\begin{remark}
{\rm
We do not expect the values $m=n-1$ or $m=n-2$ to occur for $n>4$. They have
to be included as possibilities simply because the lower bound on
$D(n)$ given by Proposition~\ref{prop:lowerbd1} 
is too weak to exclude them. It is possible to have $m > n/2$, for
example a maxdet submatrix of order $m = n/2+1$ 
occurs in Hadamard matrices of orders $n = 4$ and $n = 12$.
}
\end{remark}

We can prove a result similar to, but weaker than,
Theorem~\ref{thm:excluded_interval2} without assuming the Hadamard
conjecture.  Let the \emph{prime gap function}
$\lambda(n)$ be the maximum gap between consecutive
primes $(p_i,p_{i+1})$ with $p_i \le n$. Then we have:
\begin{theorem}	\label{thm:unconditional}
Let $H$ be a Hadamard matrix of order $n\ge 4$ having a maxdet submatrix $M$
of order $m < n$, and let $\lambda(n)$ be the prime gap function.
Then there exist positive constants $c_1$, $c_2$ such
that $m < n/2 + c_1\lambda(n)\ln n$ or 
\hbox{$m \ge n - c_2\lambda(n)$}.
\end{theorem}
\begin{proof}[Sketch of proof.]
The proof is similar to that of Theorem~\ref{thm:excluded_interval2},
but
uses Theorem~1 and Corollary~1 of~\cite{rpb249} 
in place of Proposition~\ref{prop:lowerbd1}.
Thus~\eqref{eq:ineq_weaker} is replaced by
\[m^{m/2}\left(\frac{4}{ne}\right)^{\lambda(n/2)/2} \le n^{m-n/2}d^{d/2},\]
and~\eqref{eq:ineq_fn} by
\[f(x) \le \lambda(n/2)\,\frac{\ln(ne/4)}{n}\,\raisedot\]
The remainder of the proof follows that of
Theorem~\ref{thm:excluded_interval2}, except that some of the explicit
constants have to be replaced by $O(\lambda(n/2))$ terms, and we have
to assume that $n$ is sufficiently large, say $n \ge n_0$. 
At the end, we can increase
$c_1$ or $c_2$ if necessary to ensure that
$n/2 + c_1\lambda(n)\ln n >  n - c_2\lambda(n)$ for $4 \le n < n_0$.
\end{proof}

\begin{remark}	\label{remark:BHP}
{\rm
By a result of Baker, Harman and Pintz~\cite{BHP},
$\lambda(n) = O(n^{21/40})$, so the excluded interval 
in Theorem~\ref{thm:unconditional} has length $\sim n/2$.
}
\end{remark}

\begin{remark}
{\rm
It should be possible to sharpen
Theorems~\ref{thm:excluded_interval2}--\ref{thm:unconditional}
by using the (asymptotically sharper)
bounds on $D(m)$ given in~\cite{rpb253}
instead of the bound of Proposition~\ref{prop:lowerbd1}
(this is work in progress).
}
\end{remark}

\pagebreak[3]
\section{Sequences related to minors} \label{sec:numerics}

In this section we define two sequences related to the minors
of maxdet matrices, and give the first $21$ terms in each 
sequence~\cite{OEIS}.
In the following definitions, $\N$ denotes the positive integers.

For the convenience of the reader, Tables~\arabic{spectrum}--\arabic{maxdet}
give some data taken from Orrick and Solomon~\cite{Orrick-www}, where
references to the original sources may be found.
Table~\arabic{spectrum}
gives the \emph{spectrum} of possible (absolute values of)
determinants of $\{\pm 1\}$-matrices of order $n \le 11$, normalised by
the usual factor $2^{n-1}$. 
In this and other tables, the notation ``$a..b$'' is a shorthand for
``$\{x\in\N: a\le x\le b\}$''.
Table~\arabic{maxdet} gives $\Delta(n) := D(n)/2^{n-1}$
for $n \le 21$.

\pagebreak[3]\begin{table}[ht]
\begin{center}
\begin{tabular}{|c|l|}
\hline
$n$ & $\;\;\;\;\;\;$ Spectrum $\{|\det(A)|/2^{n-1}\}$\\
\hline
1 & $\{1\}$\\
2 & $\{0, 1\}$\\
3 & $\{0, 1\}$\\
4 & $\{0..2\}$\\
5 & $\{0..3\}$\\
6 & $\{0..5\}$\\
7 & $\{0..9\}$\\
8 & $\{0..18, 20, 24, 32\}$\\
9 & $\{0..40, 42, 44, 45, 48, 56\}$\\[2pt]
  & $\{0..102, 104, 105, 108, 110, 112,$\\[-4pt]
\raisebox{5pt}{10}   & $\;\;116, 117, 120, 125, 128, 144\}$\\[2pt]
   & $\{0..268, 270..276, 278..280, 282..286,$\\[-4pt]
\raisebox{5pt}{11}   & $\;\;288, 291, 294..297, 304, 312, 315, 320\}$\\
\hline
\end{tabular}
\caption{Spectrum of $\{\pm 1\}$-matrices of order $n \le 11$,}
from Orrick and Solomon~\cite{Orrick-www};
for $n=13$ see~\cite{rpb244}.
\end{center}
\squash
\end{table}

\pagebreak[3]\begin{table}[ht]
\begin{center}
\begin{tabular}{||c|c||c|c||c|c||c|c||}
\hline
$n$ 	& $\Delta(n)$ & $n$ & $\Delta(n)$ & $n$ & 
 $\Delta(n)$ & $n$ & $\Delta(n)$\\
\hline
--    	& --   	& 1 	& 1   	& 2	& 1	& 3	& 1\\
4	& 2	& 5	& 3	& 6	& 5	& 7	& 9\\
8	& $4\times 2^3$ & 9 & $7\times 2^3$ & 
 10 & $18\times 2^3$ & 11 & $40\times 2^3$\\
 12 &	$6\times 3^5$& 13 & $15\times 3^5$ &
 14 & $39\times 3^5$ & 15 & $105\times 3^5$ \\
 16 & $8\times 4^7$  & 17 & $20\times 4^7$ &
 18 & $68\times 4^7$ & 19 & $833\times 4^6$ \\
 20 & $10\times 5^9$ & 21 & $29\times 5^9$ &
 -- & -- & -- & -- \\
\hline
\end{tabular}
\caption{Maximal determinants $\Delta(n) = D(n)/2^{n-1}$, $n \le 21$,}
from Orrick and Solomon~\cite{Orrick-www}.
\end{center}
\end{table}

We are interested in when the full spectrum
of possible minor values
occurs in the minors of maxdet matrices of given order~$n$.
\begin{definition}	\label{defn:spec_A}
The \emph{full-spectrum threshold} of an $n \times n$ $\{\pm 1\}$ matrix $A$
is the maximum $m_f\le n$ such that the full spectrum of possible
values occurs for the minors of order $m_f$ of $A$.
\end{definition}
\begin{definition}	\label{defn:spec_n}
The \emph{full-spectrum threshold} $m_f:\N \to \N$ is 
the maximum of the full-spectrum threshold of $A$ over all 
maxdet matrices $A$ of order $n$.
\end{definition}
We write $m_f(A)$ or $m_f(n)$ to denote the full-spectrum thresholds
of Definitions~\ref{defn:spec_A} or \ref{defn:spec_n} respectively;
which is meant should be clear from the context.
The values of $m_f(n)$ for $1 \le n \le 21$ are given in
Table~\arabic{depth}.
We note that the full-spectrum threshold $m_f(A)$ does depend on the 
HT-equivalence class of $A$. For example, the four HT-equivalence
classes for order~$16$ give four different 
values $m_f \in \{5,6,7,8\}$, see
Tables~\arabic{order16a}--\arabic{order16d}.

For the reasons mentioned in \S\ref{sec:intro}, we are also 
interested in the largest order
of a maxdet matrix contained as a proper submatrix of a given 
maxdet matrix. We make some definitions analogous to
Definitions~\ref{defn:spec_A}--\ref{defn:spec_n}.
\begin{definition}	\label{defn:depth_A}
The \emph{complementary depth} of an $n \times n$ $\{\pm 1\}$ matrix $A$
is the maximum $m_d < n$ such that a maxdet matrix of order $m_d$
occurs as a proper submatrix of $A$, or $0$ if $n=1$.
The \emph{depth} of $A$ is $d(A) := n-m_d(A)$.
\end{definition}
\begin{definition}	\label{defn:depth_n}
The \emph{complementary depth} $m_d:\N \to \Z$ is the maximum of the
complementary depth of $A$ over all maxdet matrices $A$ of order~$n$.
The \emph{depth} $d:\N \to \Z$ is defined by $d(n) := n - m_d(n)$.
\end{definition}
We write $d(A)$ or $d(n)$ for the depths of Definitions~\ref{defn:depth_A}
or \ref{defn:depth_n} respectively;
similarly for $m_d(A)$ and $m_d(n)$.
Clearly $d(A)$ depends on the HT-equivalence class of $A$~--
for example, see Tables~\arabic{order18a}--\arabic{order18c}
for the three HT-equivalence classes of order~$18$ with depths
$7$, $7$ and $10$.
{From} Definition~\ref{defn:depth_n}, 
$d(n)$ is the minimum of $d(A)$ over all maxdet matrices
$A$ of order~$n$, so $d(18) = 7$.
Computed values of $d(n)$, $m_d(n)$ and $m_f(n)$ 
for $1 \le n \le 21$ are given
in Table~\arabic{depth}.
It is clear from the definitions that $m_f(n) \le m_d(n)$ for $n>1$.

\begin{table}[ht]
\begin{center}
\begin{tabular}{|c|cccccccccccc|}
\hline
$n$ & 1 & 2 & 3 & 4 & 5 & 6 & 7 & 8 & 9 & $\!10\!$ & $\!11\!$ & $\!12\!$ \\
\hline
$d$ & 1 & 1 & 1 & 1 & 1 & 1 & 1 & 4 & 1 & 2 & 3 & 5 \\
$m_d$ & 0 & 1 & 2 & 3 & 4 & 5 & 6 & 4 & 8 & 8 & 8 & 7 \\
$m_f$ & 1 & 1 & 2 & 2 & 3 & 4 & 6 & 4 & 6 & 6 & 7 & 6 \\ 
\hline
\end{tabular}
~\\[5pt]
\begin{tabular}{|c|ccccccccc|}
\hline
$n$ & 13 & 14 & 15 & 16 & 17 & 18 & 19 & 20 & 21\\
\hline
$d$ & 6 & 7 & 8 & 8 &  1 & 7 & 10 & 10 & 10\\
$m_d$ & 7 & 7 & 7 & 8 & 16 & 11 & 9  & 10 & 11\\
$m_f$ & 7 & 7 & 7 & 8 & 8 & 8 & 9  & 8 & 10\\ 
\hline
\end{tabular}
\caption{Depths $d(n)$ of largest maxdet proper}
submatrices, complementary depth $m_d(n)=n-d(n)$,\\
and full-spectrum threshold $m_{f}(n)$; see
Definitions~\ref{defn:spec_n}--\ref{defn:depth_n}.
\end{center}
\end{table}

If $n \equiv 0 \pmod 8$ then $d(n) = n/2$ in the range of our
computations.
If $n \equiv 4 \pmod 8$ then both $d(n) = n/2$ (for $n=20$)
and $d(n) = n/2-1$ (for $n=4, 12$) are possible.
We see from Table~\arabic{depth} 
that the computed values all satisfy $d(n) \le (n+1)/2$.
It is interesting that the value
$d=1$ occurs for $n \le 7$, $n=9$ and $n=17$ (contrary to the remark
of Schmidt quoted in~\S\ref{sec:intro}).

\pagebreak[3]
\section{Further results and observations on minors}\label{sec:data}

In Tables \arabic{order1}--\arabic{order21e},
which may be found in the Appendix at the end of the paper,
we give computational results on the minors of maxdet
matrices of order $n \le 21$. 
Here we make some empirical observations on the results
and mention a fact (Proposition~\ref{prop:small_minors}) 
that was suggested by them.

Let $k = \lfloor n/4 \rfloor$.
Factors of the form $k^{2k},k^{2k-1},\ldots,k^2, k$ 
are present as we descend through the
minors of maxdet matrices of even order~$n$.
For Hadamard matrices this is an easy consequence of the 
Hadamard bound $n^{n/2}$ and Sz\"oll\H{o}si's theorem, 
but for $n \equiv 2 \pmod 4$
we do not have a simple explanation.
Factors of the form $k^{2k-1},k^{2k-2},\ldots,k$
are present in the minors of maxdet matrices
with order $n \equiv 1 \pmod 4$, 
and factors of the form $k^{2k-2},k^{2k-3},\ldots,k$ are present
if $n \equiv 3 \pmod 4$. 
It is an open question whether this behaviour persists for $n > 21$.
The observed divisibility properties are related to the structure
of the Gram matrices $A^TA$ of maxdet matrices $A$, but
in general this structure is unknown. For a summary of what is currently
known, see~\cite{Orrick-www}.

The presence of high powers of $k=\lfloor n/4\rfloor$ 
in the minors of order $m$, and high powers of a
possibly different integer $k'=\lfloor m/4 \rfloor$
in the maximal determinants of order $m$, gives one
explanation of why certain minors can not meet the maximal
determinant for that order.\footnote{Of course, 
this begs the question of \emph{why}
the maximal determinants are divisible by a high power of $k'$~-- we do not
have a convincing explanation for this unless the order $m$ is
such that the Hadamard, Barba~\cite{Barba33}, or 
Ehlich-Wojtas~\cite{Ehlich64a,Wojtas64}
bound is achieved, 
in which case it follows from the form of the relevant bound, 
see Osborn~\cite[pp.~98--99]{Osborn02}.}
For example, Table~\arabic{order12} shows that 
a Hadamard matrix of order $n=12$
has minors of order $11$ with scaled
value $3^5$, but
$\Delta(11) = 5\times 2^6$, which contains a high power of $2$,
not of $3$.  

Proposition~\ref{prop:mean_detsq} is from~\cite[Theorem 1]{rpb250}.
The upper bound is sharp because it is attained for Hadamard matrices.
For the case that $A$ is a Hadamard matrix, the result
is due to de Launey and Levin~\cite[Proposition~2]{LL}.
\begin{proposition}	\label{prop:mean_detsq}
Let $A$ be a square $\{\pm 1\}$ matrix of order $n \ge m > 1$. 
Then the mean value
$E(\det(M)^2) $
of $\det(M)^2$, taken over all $m\times m$ submatrices $M$ of $A$,
satisfies
\begin{equation}	\label{eq:mean_sq_upperbd}
E(\det(M)^2) \le n^m\Big/\binom{n}{m}\,.
\end{equation}
Moreover, equality holds in~\eqref{eq:mean_sq_upperbd} 
iff $A$ is a Hadamard matrix.
\end{proposition}
\begin{remark}
{\rm
For random $\{\pm1\}$ matrices $M$ of order~$m$, the expected value
of $\det(M)^2$ is $m!$, by a result of Tur\'an~\cite{Turan40}.
The last sentence of
Proposition~\ref{prop:mean_detsq} implies that $E(\det(M)^2) \ge m!$
for the order-$m$ submatrices $M$ of a Hadamard matrix,
with strict inequality if $m>1$.
}
\end{remark}
As a check on the correctness of our programs,
we computed the mean value of $\det(M)^2$
for submatrices of order $m$ of maxdet
matrices of order \hbox{$n \le 21$}, and $2 \le m \le n$.
The results agreed with
the predictions of Proposition~\ref{prop:mean_detsq}.
The following conjecture is consistent with our computations.
\begin{conjecture}	\label{conj:conj_lowerbd}
Let $A$ be a maxdet matrix.
Then the mean value $E(\det(M)^2)$ of $\det(M)^2$
taken over all $m\times m$ submatrices $M$ of $A$ satisfies the inequality
\begin{equation}	\label{eq:conj_lowerbd}
E(\det(M)^2) \ge m!
\end{equation}
Moreover, the inequality~\eqref{eq:conj_lowerbd} is strict for $m > 1$.
\end{conjecture}
Table~\arabic{bounds_15_16} gives some data for orders $13\le n\le 15$ 
to support
Conjecture~\ref{conj:conj_lowerbd}. In the table,
$R_{\rm L}(m,n)$ is the ratio of $E(\det(M)^2)$, for
submatrices $M$ of \hbox{order} $m$ of a maxdet matrix of \hbox{order}
$n$, to
the conjectured lower bound $m!$.
Similarly, $R_{\rm H}(m,n)$ is the \hbox{ratio}
of $E(\det(M)^2)$ to the upper bound~\eqref{eq:mean_sq_upperbd}.
We see that $R_{\rm L}(m,n) > 1$
for $2 \le m \le n$ (as conjectured), and the lower bound is reasonably
good for $m \le 5$, but deteriorates for larger $m$.  The upper bound is
within a factor of three of $E((\det M)^2)$ for all $m$. 
A similar pattern occurs for
all orders $n \le 21$, except that
$R_{\rm H}(m,n) = 1$ if $n$ is a Hadamard order,
in accordance with Proposition~\ref{prop:mean_detsq}.

\begin{table}[ht]	
\begin{center}
\begin{tabular}{|c|cc|cc|cc|}
\hline
$m$ & 
$R_{\rm L}(m,13)\!\!\!$ & $R_{\rm H}(m,13)$ &
$R_{\rm L}(m,14)\!\!\!$ & $R_{\rm H}(m,14)$ &
$R_{\rm L}(m,15)\!\!\!$ & $R_{\rm H}(m,15)$ \\
\hline
2 & 1.077 & 0.994	& 1.067 & 0.991	& 1.059		& 0.988\\
3 & 1.259 & 0.983	& 1.222 & 0.973	& 1.195		& 0.966\\
4 & 1.611 & 0.968	& 1.516 & 0.948	& 1.445		& 0.935\\
5 & 2.283 & 0.949	& 2.054 & 0.917	& 1.890		& 0.897\\
6 & 3.625 & 0.928	& 3.072 & 0.882	& 2.694		& 0.852\\
7 & 6.560 & 0.904	& 5.139 & 0.843	& 4.233		& 0.804\\
8 & 13.81 & 0.879	& 9.773 & 0.802	& 7.427		& 0.752\\
9 & 34.80 & 0.851	& 21.58 & 0.759	& 14.79		& 0.699\\
10 & 109.4 & 0.823	& 56.93 & 0.715	& 34.14		& 0.645\\
11 & 457.4 & 0.795	& 187.0 & 0.671	& 94.00		& 0.592\\
12 & 2864 & 0.765	& 815.6 & 0.627	& 321.7		& 0.540\\
13 & 35796 & 0.736	& 5318 & 0.584	& 1461		& 0.491\\
14 &       &		& 69137 & 0.542	& 9902		& 0.444\\
15 &       & 		& 	&	& 133638	& 0.399\\
\hline
\end{tabular}
~\\[2pt]
\caption{Ratio of $E(\det(M)^2)$ to lower and upper bounds, $13\le n\le 15$.}
For definitions of $R_{\rm L}(m,n)$ and $R_{\rm H}(m,n)$, see text.
\end{center}
\squash
\squash
\end{table}

The frequencies of occurrence
of small singular submatrices of Hada\-mard matrices are given
in the following Proposition~\cite[Corollary~$4$]{rpb250}, 
which was suggested by the computational
results before we found a proof. The case $m=2$ is implicit in a paper
of Little and Thuente~\cite[pg.~254]{LT}.
\pagebreak[3]
\begin{proposition}	\label{prop:small_minors}
Let $H$ be a Hadamard matrix of order $n$, and let $Z(m,H)$ be the number of
minors of order $m$ of $H$ that vanish.  Then
\begin{equation}	\label{eq:Z2}
Z(2,H) = n^2(n-1)(n-2)/8,  \;\;\text{and}
\end{equation}
\begin{equation}	\label{eq:Z3}
Z(3,H) = n^2(n-1)(n-2)(n-4)(5n-4)/288.
\end{equation}
\end{proposition}

\pagebreak[3]
\begin{remark}
{\rm
There is no analogue of Proposition~\ref{prop:small_minors} for minors
of order $m > 3$, because the value of $Z(m,H)$ can depend on the
HT-equivalence class of $H$, so is not given by a polynomial in~$m$
(unless $m\le 3$). 
For example, the four HT-equivalence classes
of Hadamard matrices of order $16$ have $1717520$, $1712912$,
$1710608$, and $1709456$ vanishing minors of order~$4$.
For maxdet matrices of order $n \equiv 3 \pmod 4$, 
we sometimes find different numbers of vanishing
minors of order~$2$. For example, if $n=11$, we get
$1391$, $1389$, and $1401$ vanishing minors 
for the three HT-equivalence classes. 
The right-hand-side
of eqn.~\eqref{eq:Z2}, which 
by~\cite[Corollary~3.1]{rpb250} 
is a lower bound on the number of vanishing minors in this 
non-Hadamard case, gives $1362$ (rounded up). 
}
\end{remark}

\pagebreak[3]
Finally, we briefly consider the frequencies (or \emph{multiplicities})
with which the different
values of $|\det(M)/2^{n-1}|$ occur for minors of order $m$ of a maxdet
matrix of order~$n$.  In Table~\arabic{freqs_15_7} we give the results
of computations for $m = 7$, $n=15$, which gives the typical behaviour
that we have observed.\footnote{Data on multiplicities
for other values of $m$ and $n$ may be found
at~\cite{Brent-www}.} The second column gives the observed multiplicity of
a minor with $|\det(M)|/2^{n-1}$ equal to the integer in the first column.
The third column gives the multiplicities observed when taking a random sample
of $\binom{15}{7}^2 = 41409225$ uniformly distributed $\{\pm 1\}$-matrices
of order~$m$ (we call this the \emph{random model}).  
It is clear from the table that the actual distribution is
nothing like the distribution for the random model.  
A~$\chi^2$ test gives an absurdly
small probability $< 10^{-10^{10}}$ 
that the two samples were drawn from the same distribution.
Similar behaviour occurs for other values of $m \ge 2$.  For example,
when $m=2$ we find $5187$ zero minors
and $5838$ nonzero minors, but for random matrices of order~$2$
we expect zero and nonzero values to occur with equal probability.

\begin{table}[ht]	
\begin{center}
\begin{tabular}{|c|c|c|c|}
\hline  
		 & multiplicity & multiplicity in & \\[-5pt]
\raisebox{5pt}{$|\det(M)|/2^{n-1}|$} & of minors    & random model &
\raisebox{5pt}{ratio}\\
\hline
0 	& 12857784 	& 24030613	& 0.54\\
1	& 8402100	& 11140444	& 0.75\\
2	& 10831128	& 4662108	& 2.32\\
3	& 3483909	& 924336	& 3.77\\
4	& 3935280	& 504938	& 7.79\\
5	& 622842	& 76496		& 8.14\\
6	& 927162	& 55811		& 16.61\\
7	& 129576	& 7769		& 16.68\\
8	& 201900	& 6102		& 33.09\\
9	& 17544		& 608		& 28.86\\	
\hline
total	& 41409225	& 41409225	&  \\
\hline
\end{tabular}
~\\[2pt]
\caption{Comparison of observed multiplicities of minors of}
         order~$7$ in a maxdet matrix of order~$15$
	 with a random model
\end{center}
\squash
\end{table}

Table~\arabic{freqs_15_7} shows that the normalised minors are biased
towards even values. 
For the random model, this bias can be explained by
reducing to the $\{0,1\}$ case and considering the evaluation of the
determinant in $\Z/2\Z$. Then, for large~$n$, we expect even values to occur
about $71$\% of the time.%
\footnote{The precise constant in the limit as $n \to
\infty$ is $1-\prod_{k\ge 1}(1-2^{-k})$, see~\cite{rpb094,GR,Landsberg}.}
This prediction is in agreement with the data for the random model (the
third column). For the second column we find that even values occur about
$69$\% of the time, which is close to the prediction for the random model.
Thus, although the actual distribution differs considerably from that of the
random model, the bias towards even values persists. In comparison
with the random model, the extreme bias in favour of even values in the
tail of the distribution compensates for a bias against
zero minors.

\section{Algorithms for computing the set of minors} \label{sec:alg}

Recall that ${S}_k(A)$ is the set of $k\times k$ submatrices 
of an $n \times n$ $\{\pm 1\}$-matrix $A$,
where $0 < k \le n$.  In this section we describe algorithms for
enumerating all the minors of $A$. Our application is to maxdet
matrices, but the algorithms apply to all square $\{\pm 1\}$-matrices $A$,
and with trivial changes they would also apply to
rectangular $\{\pm 1\}$-matrices.

In our enumeration we consider only the absolute values of minors, 
normalised by the factor $2^{n-1}$ which always divides the determinant
of an $n \times n$ $\{\pm 1\}$-matrix,
so the \emph{set of minors} of $A$ is defined to be
the set \[{\cal S}(A) = \cup_{k=1}^n\{|\det(M)|/2^{n-1}: M \in {S}_k(A)\}.\]

There are $\binom{n}{k}$ possible choices of rows and $\binom{n}{k}$ possible
choices of columns for a minor of order $k$,
so altogether a total of
\[T_n := \sum_{k=1}^n {\binom{n}{k}}^2 = {\binom{2n}{n}} - 1 \sim
\frac{4^n}{\sqrt{\pi n}}\]
possibilities to consider when finding the set ${\cal S}(A)$ for an $n \times n$
matrix $A$.  In this section we consider four possible algorithms
(of which we used two) for
finding ${\cal S}(A)$.  Their running times all involve the factor $4^n$, so none
of them is practical for $n$ much larger than $20$, but they differ
significantly in the factor multiplying $4^n$ and in their space
requirements.  

Each algorithm has two variants: the first just determines the set ${\cal S}(A)$ of
minors; the second also counts the multiplicity of each minor, that is, how
often a given value $d\in {\cal S}(A)$ occurs. 
We describe the first variant of each
algorithm, and briefly mention the changes required for the second variant.

In the descriptions of the algorithms we explain how to compute the complete
set ${\cal S}(A)$; it should be clear how to compute the subset of ${\cal S}(A)$
corresponding to the minors of a given order~$k$.

\subsubsection*{Algorithm A}

\emph{Algorithm A} simply considers, for each $k$ in $\{1,2,\ldots,n\}$, 
the set of all $k\times k$ submatrices of $A$, 
and evaluates the determinant of each such submatrix $M$ by Gaussian
elimination with partial pivoting, using floating-point arithmetic. The
computed determinant is scaled by division by $2^{n-1}$ and rounded to the
nearest integer.

Clearly there is a danger that rounding errors during the process of
Gaussian elimination could lead to an incorrectly rounded integer result.
However, our experiments, using IEEE standard $64$-bit floating-point 
arithmetic~\cite{IEEE}, showed that this is not a problem
for the values of $n$ that we considered ($n \le 25$). Gaussian
elimination with partial pivoting is numerically 
stable~\cite{Higham,Wilkinson65}, 
and the maximum scaled determinant
of a $25 \times 25$ $\{\pm 1\}$-matrix is $42 \times 6^{11} = 15237476352$,
meeting the Barba bound~\cite{Barba33,Orrick-www}, so only requires
$34$ bits of precision, significantly fewer than the $53$ bits
provided by IEEE standard arithmetic.  As a precaution, our implementation
prints a warning and halts if the fractional part of a scaled determinant
exceeds $1/8$; this never occurred for $n \le 25$.

Gaussian elimination requires $O(k^3)$ arithmetic operations to evaluate
the determinant of a $k\times k$ matrix. As is traditional in numerical
analysis, we count multiplications/divisions but ignore
additions/subtractions. With this convention, Gaussian elimination requires
$k^3/3 + O(k^2)$ operations. Thus the total cost is
\[W_A \sim \sum_{k=1}^n {\binom{n}{k}}^2\frac{k^3}{3} 
 \sim \frac{4^n n^{5/2}}{24\sqrt{\pi}}\,.\]
The storage requirements of Algorithm~A are minimal, apart from the space
required to store the results, that is the set ${\cal S}(A)$ of minors and (if
required) their multiplicities. This is common to all the algorithms
considered~-- they all need space to store their results.

The set ${\cal S}(A)$ can be represented using one bit for each possible
value $|\det|/2^{n-1}$.  From the Hadamard bound, this requires at most
$2^{1-n}n^{n/2}+1$ bits. For example, if $n=20$, it requires $19531251$ bits
(2.33~MB). For $n=24$ it
requires 519~MB, which is still feasible. For the variant that counts
multiplicities, each bit needs to be replaced by an integer word (say $32$
bits or $4$ bytes), so the storage required would be acceptable for
$n = 20$ (75~MB) but excessive for $n=24$ ($16$~GB), since the computers
available to us typically have memories of $1$ to $4$~GB.

Fortunately, a much more economical representation of ${\cal S}(A)$ is usually
possible, because not all minors
in the range $[0,\lfloor 2^{1-n}n^{n/2}\rfloor]$ actually occur.  
The set ${\cal S}(A)$ is usually quite sparse, especially when the order $n$ of $A$ 
is divisible by $4$.
For example, with Hadamard matrices of order $16$, 
we have $\#{\cal S}(A) < 100$ (see \S\ref{sec:data}). Thus,
instead of using $131073$ bits to represent ${\cal S}(A)$, 
we can use a hash table
with say $200$ words~\cite{Knuth3}. With such an implementation, the
storage requirements are moderate for $n\le 25$. 

\subsubsection*{Algorithm B}

Algorithm B is similar to Algorithm A, but uses a rank-$1$ updating formula
to update the inverse and determinant of each $k \times k$ submatrix $B$
of $A$ if we already know the inverse and determinant of a submatrix that
differs from $B$ in only one row or column.  The inverse updating formula 
\[(B + uv^T)^{-1} = B^{-1} - (1 + v^TB^{-1}u)^{-1}B^{-1}uv^TB^{-1}\]
is known at the \emph{Sherman-Morrison} formula~\cite{SM}~-- the determinant
updating formula 
\[\det(B + uv^T) = \det(B)(1 + v^TB^{-1}u)\]
seems to be known only as a ``matrix determinant lemma''.

Since the updating steps require $\sim k^2$ operations, the complexity is 
\[W_B \sim {4^n n^{3/2}}/{\sqrt{\pi}}\,,\] 
so $W_A/W_B \sim n/24$.

We did not use Algorithm~B because the constant factors involved make it
slower than Algorithm~A for $n \le 20$, and because it is difficult to
guarantee a correctly rounded integer result due to possible numerical 
instability.  Also, even with exact arithmetic, 
we would have to use a different method whenever $B$ is singular.

\subsubsection*{Algorithm C}

Algorithm C uses integer arithmetic and evaluates each $k \times k$
determinant using the method (attributed to Laplace) 
of expansion by minors, see for example~\cite[Chapter 4]{Muir}.  
To compute a $k\times k$ minor $\det(B)$, 
we need to know the $(k-1)\times (k-1)$ minors formed by deleting
the first row and an arbitrary column of $B$. If these minors have been
saved from a previous computation, then the work involved in computing
one minor $|\det(B)|$ is only $k$ multiplications (by $\pm 1$) and $k-1$
integer additions, plus any overheads involved in retrieving the previously
stored values.  If we assume, for purposes of comparison with Algorithms A--B, 
that the work involved amounts to $k$ operations, then the total cost is
\[W_C = \sum_{k=1}^n \binom{n}{k}^2 k \sim 
 \frac{4^n n^{1/2}}{2\sqrt{\pi}}\,,\]
giving $W_A/W_C \sim n^2/12$.

Unfortunately, this algorithm has a potentially large memory requirement.
If we compute the minors of order $k$ in increasing order 
$k = 1, 2, \ldots, n$, then to compute the minors of order $k$
we need all $\binom{n}{k-1}^2$ minors of order $k-1$.  In the worst
case, when $k \approx n/2$, the memory required to store the minors
of order $k-1$ and $k$ is about $4^{n+1}/(\pi n)$ words, which is too large
to be practical for the values of $n$ that we wish to consider. More
memory-efficient implementations are possible, but complicated. For this
reason we discarded Algorithm~C and implemented a slightly slower, but much
simpler algorithm, Algorithm~D. 

\subsubsection*{Algorithm D}

The idea of Algorithm D is the same as that of Algorithm C.  However, when
computing the minors of order $k$ of an $n\times n$ matrix $A$, 
the outer loop runs over all $\binom{n}{k}$ combinations of $k$ rows of
$A$.  Having selected these $k$ rows, forming a $k \times n$ matrix $B$,
we now compute all minors of order $k$ of $B$. At the $j$-th step we
compute all minors of order $j$ in the last $j$ rows of $B$.  Thus the
number of operations (counting as for Algorithm~C) to compute the
minors of order $k$ of $B$ is
\[\sum_{j=1}^k \binom{n}{j}j\]
and the space requirement is at most 
$2\binom{n}{\lfloor n/2 \rfloor} \sim 2^{n+\frac32}/\sqrt{\pi n}$ words,
much less than for Algorithm~C.
The overall operation count is
\[\sum_{k=1}^n \binom{n}{k} \sum_{j=1}^k \binom{n}{j}j = 4^{n-1}n.\]
This is larger than the operation count for Algorithm~C by a
factor\linebreak
$\sim \sqrt{\pi n}/2$, but smaller than the operation count for
Algorithm~A by a factor $\sim n^{3/2}/(6\sqrt{\pi})$.

If a parallel implementation is desired, then it is easy to parallelise over
the outer loop~-- different processors can work on different combinations
of $k$ rows of $A$ in parallel.

We ran both Algorithms A and D on small cases to check the correctness of
our implementations. For the large cases we used mainly Algorithm~D,
which is much faster than Algorithm~A for the most time-consuming
cases ($k \approx n/2$), as expected from the operation counts given
above. 

\subsubsection*{Acknowledgements}
We thank Will Orrick for his advice and assistance in locating several
of the references, and Warren Smith for pointing out the connection between
Jacobi and Sz\"oll\H{o}si. We also thank an anonymous referee for
useful suggestions and corrections.

\pagebreak[3]

\bigskip
\hrule
\bigskip

\noindent 2010 {\it Mathematics Subject Classification}:
Primary 05B20; Secondary 15B34, 62K05.

\noindent \emph{Keywords: } 
Hadamard matrices, maximal determinant matrices, 
saturated D-optimal designs, minors.

\bigskip
\hrule
\bigskip

\noindent (Concerned with sequences
\seqnum{A215644},\seqnum{A215645},\seqnum{A003432},\seqnum{A003433}.) 

\bigskip
\hrule
\bigskip

\vspace*{+.1in}
\noindent

\bigskip
\hrule
\bigskip

\noindent

\pagebreak[3]
\subsection*{Appendix: Tables of minors for orders $\le 21$}

In Tables \arabic{order1}--\arabic{order21e}, 
$k = \lfloor n/4 \rfloor$, where $n$ is the order of
the $\{\pm 1\}$-matrix.  The first column gives the order $m$ of the minor,
$1 \le m \le n$. The second column gives the set of absolute values
of the minors of order $m$, divided by the known factor $2^{m-1}$.   
In this column the notation 
``$\{a,b,\ldots\}\times k^\alpha$'' is a shorthand for\linebreak
``$\{ak^\alpha, bk^\alpha,\ldots\}$'', etc.
For $n \in \{19,21\}$ we have abbreviated the entry in the second column by
giving only the minimum and maximum rather than the complete set,
using ``$({\rm min, max})$'' instead of ``$\{a,b,\ldots\}$''.
In such cases we write ``$(a,b)\times k^\alpha$'' 
instead of ``$(ak^\alpha,bk^\alpha)$''.

In the third column we give the scaled maximum
determinant $\Delta(m) = D(m)/2^{m-1}$ (redundant, but included for easy
comparison with the entries in the second column).
The fourth column answers whether some minor meets the maximum possible
determinant for its order (see Table~\arabic{maxdet}), 
and the last column answers
whether the full spectrum of possible values of minors
(as given in Table~\arabic{spectrum}) occurs.

If there is more than one HT-equivalence class for an order~$n$,
the classes are listed in the same order as they are given
in~\cite{Orrick-www}.
The information given in Tables \arabic{order1}--\arabic{order21e} is
sufficient to uniquely identify each HT-equivalence class.

To avoid making this Appendix
excessively long, we have omitted details such as the frequency
of occurrence of each minor value. Further information is available
on our website~\cite{Brent-www}.

\pagebreak[4]
\begin{table}[ht]		
\begin{center}
{\bf Orders 1--6}\\[20pt]
\begin{tabular}{||c|c|c|c|c||}
\hline
$m$ & $\{{\rm minors}\}$ & $\Delta(m)$ & max? & full?\\
\hline
1 & $\{1\}$ & 1 & yes &yes\\
\hline
\end{tabular}
\\[5pt]
Table 6: $n=1$, $k=0$, $m_f = 1$\\[10pt]
\begin{tabular}{|c|c|c|c|c|}
\hline
$m$ & $\{{\rm minors}\}$ & $\Delta(m)$ & max? & full?\\
\hline
2 & $\{1\}$ & 1 & yes &no\\
1 & $\{1\}$ & 1 & yes &yes\\
\hline
\end{tabular}
\\[5pt]
Table 7: $n=2$, $k=0$, $d=1$, $m_f = 1$\\[10pt]
\begin{tabular}{|c|c|c|c|c|}
\hline
$m$ & $\{{\rm minors}\}$ & $\Delta(m)$ & max? & full?\\
\hline
3 & $\{1\}$ & 1 & yes &no\\
1--2 & full spectrum & 1 & yes &yes\\
\hline
\end{tabular}
\\[5pt]
Table 8: $n=3$, $k=0$, $d=1$, $m_f=2$\\[10pt]
\begin{tabular}{|c|c|c|c|c|}
\hline
$m$ & $\{{\rm minors}\}$ & $\Delta(m)$ & max? & full?\\
\hline
4 & $\{2\}$ & 2 & yes &no\\
3 & $\{1\}$ & 1 & yes &no\\
1--2 & full spectrum & 1 & yes &yes\\
\hline
\end{tabular}
\\[5pt]
Table 9: $n=4$, $k=1$, $d=1$, $m_f=2$\\[10pt]
\begin{tabular}{|c|c|c|c|c|}
\hline
$m$ & $\{{\rm minors}\}$ & $\Delta(m)$ & max? & full?\\
\hline
5 & $\{3\}$ & 3 & yes &no\\
4 & $\{1,2\}$ & 2 & yes &no\\
1--3 & full spectrum & 1 & yes &yes\\
\hline
\end{tabular}
\\[5pt]
Table 10: $n=5$, $k=1$, $d=1$, $m_f=3$\\[10pt]
\begin{tabular}{|c|c|c|c|c|}
\hline
$m$ & $\{{\rm minors}\}$ & $\Delta(m)$ & max? & full?\\
\hline
6 & $\{5\}$ & 5 & yes &no\\
5 & $\{1..3\}$ & 3 & yes &no\\
1--4 & full spectrum & $\le 2$  & yes &yes\\
\hline
\end{tabular}
\\[5pt]
Table 11: $n=6$, $k=1$, $d=1$, $m_f=4$\\
\end{center}
\squash
\end{table}

\pagebreak[4]
\begin{table}[ht]
\begin{center}
{\bf Orders 7--11(a)}\\[10pt]
\begin{tabular}{|c|c|c|c|c|}
\hline
$m$ & $\{{\rm minors}\}$ & $\Delta(m)$ & max? & full?\\
\hline
7 & $\{9\}$ & 9 & yes &no\\
1--6 & full spectrum & $\le 5$ & yes &yes\\
\hline
\end{tabular}
\\[4pt]
Table 12: $n=7$, $k=1$, $d=1$, $m_f=6$\\[9pt]
\begin{tabular}{|c|c|c|c|c|}
\hline
$m$ & $\{{\rm minors}\}$ & $\Delta(m)$ & max? & full?\\
\hline
8 & $\{2\}\times 2^4$ & $2\times 2^4$ & yes &no\\
7 & $\{1\}\times 2^3$ & 9 & no & no\\
6 & $\{0,1\}\times 2^2$ & 5 & no & no\\
5 & $\{0,1\}\times 2^1$ & 3 & no & no\\
1--4 & full spectrum & $\le 2$ & yes &yes\\
\hline
\end{tabular}
\\[4pt]
Table 13: $n=8$, $k=2$, $d=4$, $m_f=4$\\[9pt]
\begin{tabular}{|c|c|c|c|c|}
\hline
$m$ & $\{{\rm minors}\}$ & $\Delta(m)$ & max? & full?\\
\hline
9 & $\{7\}\times 2^3$ & $7\times 2^3$ & yes &no\\
8 & $\{2,3,4,6,8\}\times 2^2$ & $8 \times 2^2$ & yes &no\\
7 & $\{0..4\}\times 2^1$ & 9 & no & no\\
1--6 & full spectrum & $\le 5$ & yes &yes\\
\hline
\end{tabular}
\\[4pt]
Table 14: $n=9$, $k=2$, $d=1$, $m_f=6$\\[9pt]
\begin{tabular}{|c|c|c|c|c|}
\hline
$m$ & $\{{\rm minors}\}$ & $\Delta(m)$ & max? & full?\\
\hline
10 & $\{9\}\times 2^4$ & $9 \times 2^4$ & yes &no\\
9 & $\{3,6\}\times 2^3$ & $7 \times 2^3$ & no & no\\
8 & $\{0..5,8\}\times 2^2$ & $8\times 2^2$ & yes &no\\
7 & $\{0..4\}\times 2^1$ & 9 & no & no\\
1--6 & full spectrum & $\le 5$ & yes &yes\\
\hline
\end{tabular}
\\[4pt]
Table 15: $n=10$, $k=2$, $d=2$, $m_f=6$\\[9pt]
\begin{tabular}{|c|c|c|c|c|}
\hline
$m$ & $\{{\rm minors}\}$ & $\Delta(m)$ & max? & full?\\
\hline
11 & $\{20\}\times 2^4$ & $20\times 2^4$ & yes &no\\
10 & $\{0, 2..8, 10, 12, 16\}\times 2^3$ & $18 \times 2^3$ & no & no\\
 & $\{0,2..4,6,8..10,12..24,$ &&&\\[-4pt]
\raisebox{5pt}{9}
&\,\,\,\,$26..33,36,40,48\}$ & \raisebox{5pt}{56} & \raisebox{5pt}{no} & 
 \raisebox{5pt}{no}\\
8 & $\{0..18,20,24\}$& 32 & no & no\\
1--7 & full spectrum & $\le 9$ & yes &yes\\
\hline
\end{tabular}
\\[4pt]
Table 16: $n=11$(a), $k=2$, $d=4$, $m_f=7$
\end{center}
\squash
\end{table}

\pagebreak[4]
\begin{table}[ht]
\begin{center}
{\bf Orders 11(b)--12}\\[20pt]
\begin{tabular}{|c|c|c|c|c|}
\hline
$m$ & $\{{\rm minors}\}$ & $\Delta(m)$ & max? & full?\\
\hline
11 & $\{20\}\times 2^4$ & $20\times 2^4$ & yes &no\\
10 & $\{0, 1, 4..6, 8..11, 14, 16\}\times 2^3$ & $18 \times 2^3$ & no & no\\
& $\{0,2..4,6,8..24,26..28,$ &&&\\[-4pt]
\raisebox{5pt}{9}
&$30..32,36,40,44,48\}$ & \raisebox{5pt}{56} & \raisebox{5pt}{no} & 
\raisebox{5pt}{no}\\
8 & $\{0..18,20,24\}$& 32 & no & no\\
1--7 & full spectrum & $\le 9$ & yes &yes\\
\hline
\end{tabular}
\\[5pt]
Table 17: $n=11$(b), $k=2$, $d=4$, $m_f=7$\\[20pt]
\begin{tabular}{|c|c|c|c|c|}
\hline
$m$ & $\{{\rm minors}\}$ & $\Delta(m)$ & max? & full?\\
\hline
11 & $\{20\}\times 2^4$ & $20\times 2^4$ & yes &no\\
10 & $\{4, 6, 8, 12, 16\}\times 2^3$ & $18 \times 2^3$ & no & no\\
9 & $\{0..4, 6, 8, 12\}\times 2^2$ & $14 \times 2^2$ & no & no\\
8 & $\{0..8, 12, 16\}\times 2^1$ & $16 \times 2^1$ & yes &no\\
7 & $\{0..8\}$ & 9 & no & no\\
1--6 & full spectrum & $\le 5$ & yes &yes\\
\hline
\end{tabular}
\\[5pt]
Table 18: $n=11$(c), $k=2$, $d=3$, $m_f=6$\\[20pt]
\begin{tabular}{|c|c|c|c|c|}
\hline
$m$ & $\{{\rm minors}\}$ & $\Delta(m)$ & max? & full?\\
\hline
12 & $\{2\}\times 3^6$ & $2 \times 3^6$ & yes &no\\
11 & $\{1\}\times 3^5$ & $5 \times 2^6$ & no & no\\
10 & $\{0,1\}\times 3^4$ & $3^2 \times 2^4$ & no & no\\
9 & $\{0,1\}\times 3^3$ & $7 \times 2^3$ & no & no\\
8 & $\{0..2\}\times 3^2$ & $2^5$ & no & no\\
7 & $\{0..3\}\times 3^1$ & $3^2$ & yes &no\\
1--6 & full spectrum & $\le 5$ & yes &yes\\
\hline
\end{tabular}
\\[5pt]
Table 19: $n=12$, $k=3$, $d=5$, $m_f=6$\\
\end{center}
\squash
\end{table}

\pagebreak[4]
\begin{table}[ht]
\begin{center}
{\bf Orders 13--15}\\[20pt]
\begin{tabular}{|c|c|c|c|c|}
\hline
$m$ & $\{{\rm minors}\}$ & $\Delta(m)$ & max? & full?\\
\hline
13 & $\{5\}\times 3^6$ & $5 \times 3^6$ & yes &no\\
12 & $\{2,3\}\times 3^5$ & $6 \times 3^5$ & no & no\\
11 & $\{0..3\}\times 3^4$ & $5 \times 2^6$ & no & no\\
10 & $\{0..4\}\times 3^3$ & $3^2 \times 2^4$ & no & no\\
9 & $\{0..5\}\times 3^2$ & $7 \times 2^3$ & no & no\\
8 & $\{0..6\}\times 3^1$ & $2^5$ & no & no\\
1--7 & full spectrum & $\le 9$ & yes &yes\\
\hline
\end{tabular}
\\[5pt]
Table 20: $n=13$, $k=3$, $d=6$, $m_f=7$\\[20pt]
\begin{tabular}{|c|c|c|c|c|}
\hline
$m$ & $\{{\rm minors}\}$ & $\Delta(m)$ & max? & full?\\
\hline
14 & $\{13\}\times 3^6$ & $13\times 3^6$ & yes &no\\
13 & $\{4, 6, 7, 9\}\times 3^5$ & $15 \times 3^5$ & no & no\\
12 & $\{0..7,9,10\}\times 3^4$& $18 \times 3^4$ & no & no\\
11 & $\{0..9,11\}\times 3^3$& $5\times 2^6$ & no & no\\
10 & $\{0..13\}\times 3^2$& $3^2\times 2^4$ & no & no\\
9 & $\{0..15\}\times 3^1$&$7\times 2^3$&no & no\\
8 & $\{0..18,20\}$& $2^5$ & no & no\\
1--7 & full spectrum & $\le 9$ & yes &yes\\
\hline
\end{tabular}
\\[5pt]
Table 21: $n=14$, $k=3$, $d=7$, $m_f=7$\\[20pt]
\begin{tabular}{|c|c|c|c|c|}
\hline
$m$ & $\{{\rm minors}\}$ & $\Delta(m)$ & max? & full?\\
\hline
15 & $\{35\}\times3^6$& $35 \times 3^6$& yes &no\\
14 & $\{7,8,12,14,17,18,21,23,27\}\times 3^5$ & $39\times 3^5$ & no & no\\
13 & $\{0..21,23,24,26,27\} \times 3^4$&$45 \times 3^4$&no & no\\
12 & $\{0..22,24..27\} \times 3^3$ &$54 \times 3^3$&no & no\\
11 & \hspace{2mm}$ \{0..29,31,35\} \times 3^2$ & $5\times 2^6$ & no & no\\
10 & \hspace{2mm}$ \{0..36,39,40\} \times 3^1$ & $3^2\times 2^4$ & no & no\\
9 & $\{0..36, 38..40, 42, 44, 45\}$ &56&no & no\\
8 & $\{0..18,20,24\} $ & 32 & no & no\\
1--7 & full spectrum & $\le 9$ & yes &yes\\
\hline
\end{tabular}
\\[5pt]
Table 22: $n=15$, $k=3$, $d=8$, $m_f=7$
\end{center}
\squash
\end{table}

\pagebreak[4]
\begin{table}[ht]
\begin{center}
{\bf Order 16}\\[18pt]
\begin{tabular}{|c|c|c|c|c|}
\hline
$m$ & $\{{\rm minors}\}$ & $\Delta(m)$ & max? & full?\\
\hline
16 & $\{2\}\times 4^8$&$2 \times 4^8$&yes &no\\
15 & $\{1\}\times 4^7$& $35 \times 3^6$& no & no\\
14 & $\{0, 1\}\times 4^6$ & $13\times 3^6$ & no & no\\
13 & $\{0,1\}\times 4^5$ &$5 \times 3^6$&no & no\\
12 & $\{0..2\}\times 4^4$&$2 \times 3^6$& no & no\\
11 & $ \{0..3\}\times 4^3$ & $5\times 2^6$ & no & no\\
10 & $\{0..2\}\times 2^5$ & $9\times 2^4$ & no & no\\
9 & $ \{0..2\}\times 2^4$ &$7\times 2^3$& no & no\\
8 & $\{0..4\}\times 2^3$ & 32 & yes &no\\
7 & $\{0..2\}\times 2^2$ & 9 & no & no\\
6 & $ \{0..2\}\times 2^1$ & 5 & no & no\\
1--5 & full spectrum & $\le 3$ & yes &yes\\
\hline
\end{tabular}
\\[5pt]
Table 23: $n=16$(a), $k=4$, $d=8$, $m_f=5$\\[10pt]
\begin{tabular}{|c|c|c|c|c|}
\hline
$m$ & $\{{\rm minors}\}$ & $\Delta(m)$ & max? & full?\\
\hline
11--16 & as for $16$(a) & -- & -- & no\\
10 & $\{0..5\}\times 4^2$ & $9\times 4^2$ & no & no\\
9 & $\{0..4\}\times 2^3$ &$7\times 2^3$& no & no\\
8 & $\{0..6,8\}\times 2^2 $ & 32 & yes &no\\
7 & $\{0..4\}\times 2^1$ & 9 & no & no\\
1--6 & full spectrum & $\le 5$ & yes &yes\\
\hline
\end{tabular}
\\[5pt]
Table 24: $n=16$(b), $k=4$, $d=8$, $m_f=6$\\[10pt]
\begin{tabular}{|c|c|c|c|c|}
\hline
$m$ & $\{{\rm minors}\}$ & $\Delta(m)$ & max? & full?\\
\hline
10--16 & as for $16$(b) & -- & -- & no\\
9 & $\{0..9\}\times 4^1$ &$14 \times 4^1$& no & no\\
8 & $\{0..10,12,16\}\times 2^1 $ & $16\times 2^1$ & yes &no \\
1--7 & full spectrum & $\le 9$ & yes &yes\\
\hline
\end{tabular}
\\[5pt]
Table 25: $n=16$(c), $k=4$, $d=8$, $m_f=7$\\[10pt]
\begin{tabular}{|c|c|c|c|c|}
\hline
$m$ & $\{{\rm minors}\}$ & $\Delta(m)$ & max? & full?\\
\hline
9--16 & as for $16$(c) & -- & -- & no\\
1--8 & full spectrum & $\le 32$ & yes &yes\\
\hline
\end{tabular}
\\[5pt]
Table 26: $n=16$(d), $k=4$, $d=8$, $m_f=8$
\end{center}
\squash
\end{table}

\pagebreak[4]
\begin{table}[ht]
\begin{center}
{\bf Order 17}\\[20pt]
\begin{tabular}{|c|c|c|c|c|}
\hline
$m$ & $\{{\rm minors}\}$ & $\Delta(m)$ & max? & full?\\
\hline
17 &$\{5 \}\times 4^8$& $5 \times 4^8$ & yes &no\\ 
16 & $\{2, 3, 8\}\times 4^7$&$8 \times 4^7$&yes &no\\
15 & $\{0..4\}\times 4^6$& $35 \times 3^6$& no & no\\
14 & $\{0..4\}\times 4^5$ & $13\times 3^6$ & no & no\\
13 & $\{0..7\}\times 4^4$ &$5 \times 3^6$& no & no \\
12 & $\{0..9\}\times 4^3$&$2 \times 3^6$& no & no \\
11 & $\{0..13,15\}\times 4^2$ & $20\times 4^2$ & no & no\\
10 & $\{0..21,24,27\}\times 4^1$ & $36\times 4^1$ & no & no\\
9 & $\{0..40,42,44,45,48\}$ &56& no & no\\
1--8 & full spectrum & $\le 32$ & yes &yes\\
\hline
\end{tabular}
\\[5pt]
Table 27: $n=17$(a), $k=4$, $d=1$, $m_f=8$\\[20pt]
\begin{tabular}{|c|c|c|c|c|}
\hline
$m$ & $\{{\rm minors}\}$ & $\Delta(m)$ & max? & full?\\
\hline
10--17 & as for $17$(a) & -- & -- & no\\
9 & $\{0..22,24\}\times 2^1$ &$28\times 2^1$& no & no\\
1--8 & full spectrum & $\le 32$ & yes &yes\\
\hline
\end{tabular}
\\[5pt]
Table 28: $n=17$(b), $k=4$, $d=1$, $m_f=8$\\[20pt]
\begin{tabular}{|c|c|c|c|c|}
\hline
$m$ & $\{{\rm minors}\}$ & $\Delta(m)$ & max? & full?\\
\hline
11--17 & as for $17$(a) & -- & -- & no\\
10 & $\{0..10,12\}\times 2^3 $ & $18\times 2^3$ & no & no\\
9 & $\{0..12\}\times 2^2 $ &$14\times 2^2$& no & no\\
8 & $\{0..10,12,16\}\times 2^1 $ & $16\times 2^1$ & yes &no\\
1--7 & full spectrum & $\le 9$ & yes &yes\\
\hline
\end{tabular}
\\[5pt]
Table 29: $n=17$(c), $k=4$, $d=1$, $m_f=7$
\end{center}
\squash
\end{table}

\pagebreak[4]
\begin{table}[ht]
\begin{center}
{\bf Order 18}\\[20pt]
\begin{tabular}{|c|c|c|c|c|}
\hline
$m$ & $\{{\rm minors}\}$ & $\Delta(m)$ & max? & full?\\
\hline
18 & $\{17\} \times 4^8$ & $17\times4^8$ & yes &no\\
17 &$\{6, 7, 10, 11\}\times 4^7$& $20 \times 4^7$ & no & no\\ 
16 & $\{0..8, 10, 11, 13\}\times 4^6 $&$32 \times 4^6$&no & no\\
15 & $\{0..16\}\times 4^5$& $35 \times 3^6$& no & no\\
14 & $\{0..20\}\times 4^4$ & $13\times 3^6$ & no & no\\
13 & $\{0..24,26,28\}\times 4^3$ &$5 \times 3^6$& no & no \\
12 & $\{0..38,40,41,44,52\}\times 4^2$&$2 \times 3^6$& no & no \\
11 & $\{0..62,64,68,80\}\times 4^1$ & $80 \times 4^1$ & yes &no\\
  & $\{0..94,96..98,100..102,$&&&\\[-4pt]
\raisebox{5pt}{10} &$104,108,112,128\} $ & \raisebox{5pt}{144} 
 & \raisebox{5pt}{no} & \raisebox{5pt}{no}\\
9 & $\{0..40,42,44,45,48\}$ &56& no & no\\
1--8 & full spectrum & $\le 32$ & yes &yes\\
\hline
\end{tabular}
\\[5pt]
Table 30: $n=18$(a), $k=4$, $d=7$, $m_f=8$\\[15pt]
\begin{tabular}{|c|c|c|c|c|}
\hline
$m$ & $\{{\rm minors}\}$ & $\Delta(m)$ & max? & full?\\
\hline
14--18 & as for $18$(a) & -- & -- & no\\
13 & $\{0..24,26,28,32\}\times 4^3$ &$5 \times 3^6$& no & no \\
12 & $\{0..38,40,41,44,64\}\times 4^2$ &$2 \times 3^6$& no & no \\
11 & $\{0..62,64,68,80\}\times 4^1$ & $80 \times 4^1$ & yes &no\\
10 & $\{0..52,54,56,64\}\times 2^1$ & $72\times 2^1$ & no & no\\
9 & $\{0..40,42,44,45,48\}$ &56& no & no\\
1--8 & full spectrum & $\le 32$ & yes &yes\\
\hline
\end{tabular}
\\[5pt]
Table 31: $n=18$(b), $k=4$, $d=7$, $m_f=8$\\[15pt]
\begin{tabular}{|c|c|c|c|c|}
\hline
$m$ & $\{{\rm minors}\}$ & $\Delta(m)$ & max? & full?\\
\hline
14--18 & as for $18$(a) & -- & -- & no\\
13 & $\{0..24, 32\}\times 4^3$ &$5 \times 3^6$& no & no \\
12 & $\{0..37,40, 64\}\times 4^2$&$2 \times 3^6$& no & no \\
11 & $\{0..32\}\times 2^3$ & $40 \times 2^3$ & no & no\\
10 & $\{0..28,32\}\times 2^2$ & $36\times 2^2$ & no & no\\
9 & $\{0..22,24\}\times 2^1$ &$28\times 2^1$& no & no\\
1--8 & full spectrum & $\le 32$ & yes &yes\\
\hline
\end{tabular}
\\[5pt]
Table 32: $n=18$(c), $k=4$, $d=10$, $m_f=8$
\end{center}
\squash
\end{table}

\pagebreak[4]
\begin{table}[ht] 		
\begin{center}
{\bf Order 19}\\[20pt]
\begin{tabular}{|c|c|c|c|c|}
\hline
$m$ & (min, max) $|$minor$|$ & $\Delta(m)$ & max? & full?\\
\hline
19 & $(833, 833)\times 4^6$ & $833\times 4^6$ & yes &no\\
18 & $(140,784)\times 4^5$ & $1088\times4^5$ & no & no\\
17 &$(0,672)\times 4^4$& $1280 \times 4^4$ & no & no\\ 
16 & $(0, 676)\times 4^3$&$2048 \times 4^3$&no & no\\
15 & $(0, 1050)\times 4^2$& $35 \times 3^6$ & no & no\\
14 & $(0,1470)\times 4^1$ & $13\times 3^6$ & no & no\\
13 & $(0,1904)$ &$3645$& no & no \\
12 & $(0,756)$&$1458$& no & no \\
11 & $(0,312)$ & $320$ & no & no\\
10 & $(0,128)$ & 144 & no & no\\
1--9 & full spectrum & $\le 56$ & yes &yes\\
\hline
\end{tabular}
\\[5pt]
Table 33: $n=19$(a), $k=4$, $d=10$, $m_f=9$\\[20pt]
\begin{tabular}{|c|c|c|c|c|}
\hline
$m$ & (min, max) $|$minor$|$ & $\Delta(m)$ & max? & full?\\
\hline
19 & $(833, 833)\times 4^6$ & $833\times 4^6$ & yes &no\\
18 & $(168,616)\times 4^5$ & $1088\times4^5$ & no & no\\
17 &$(0,672)\times 4^4$& $1280 \times 4^4$ & no & no\\ 
16 & $(0, 740)\times 4^3$&$2048 \times 4^3$&no & no\\
15 & $(0, 1024)\times 4^2$& $35 \times 3^6$ & no & no\\
14 & $(0,1536)\times 4^1$ & $13\times 3^6$ & no & no\\
13 & $(0,2048)$ &$3645$& no & no \\
12 & $(0,1024)$ &$1458$& no & no \\
11 & $(0,288)$ & $320$ & no & no\\  
10 & $(0,128)$ & 144 & no & no\\
1--9 & full spectrum & $\le 56$ & yes &yes\\
\hline
\end{tabular}
\\[5pt]
Table 34: $n=19$(b), $k=4$, $d=10$, $m_f=9$\\[20pt]
\begin{tabular}{|c|c|c|c|c|}	
\hline
$m$ & (min, max) $|$minor$|$ & $\Delta(m)$ & max? & full?\\
\hline
14--19 & as for 19(b) & -- & -- & -- \\
13 & $(0,2560)$ &$3645$& no & no \\
1--12 & as for 19(b) & -- & -- & -- \\
\hline
\end{tabular}
\\[5pt]
Table 35: $n=19$(c), $k=4$, $d=10$, $m_f=9$
\end{center}
\squash
\end{table}

\pagebreak[4]
\begin{table}[ht]
\begin{center}
{\bf Order 20}\\[20pt]
\begin{tabular}{|c|c|c|c|c|}
\hline
$m$ & $\{{\rm minors}\}$ & $\Delta(m)$ & max? & full?\\
\hline
20 & $\{2\}\times 5^{10}$ & $2\times 5^{10}$ & yes &no\\
19 & $\{1\}\times 5^9$ & $833\times 4^6$ & no &no\\
18 & $\{0,1\}\times 5^8$ & $17\times4^8$ & no & no\\
17 &$\{0,1\}\times 5^7$& $5 \times 4^8$ & no & no\\ 
16 & $\{0..2\}\times 5^6$&$2 \times 4^8$&no & no\\
15 & $\{0..3\}\times 5^5$& $35 \times 3^6$ & no & no\\
14 & $\{0..5\}\times 5^4$ & $13\times 3^6$ & no & no\\
13 & $\{0..9\}\times 5^3$ &$5 \times 3^6$& no & no \\
12 & $\{0..18,20,24,32\}\times 5^2$&$2 \times 3^6$& no & no \\
11 & $\{0..40,42,44,48\}\times 5^1$ & $64\times 5^1$ & no & no\\
10 & $\{0..92,95,96,100,104,108,112,125,144\}$ & 144 & yes & no\\
9 & $\{0..40,42,44,48\}$ &56& no &no\\
1--8 & full spectrum & $\le 32$ & yes &yes\\
\hline
\end{tabular}
\\[5pt]
Table 36: $n=20$(a), $k=5$, $d=10$, $m_f=8$\\[20pt]
\begin{tabular}{|c|c|c|c|c|}		
\hline
$m$ & $\{{\rm minors}\}$ & $\Delta(m)$ & max? & full?\\
\hline
12--20 & as for $20$(a) & -- & -- & no\\
11 & $\{0..40,42,44,45,48\}\times 5^1$ & $64\times 5^1$ & no & no\\
   & $\{0..90,92,93,95..102,104,108,$ &  &  & \\[-4pt]
\raisebox{5pt}{10} & $  112,117,120,125,128,144\}$ 
 & \raisebox{5pt}{144} & \raisebox{5pt}{yes} & \raisebox{5pt}{no}\\
9 & $\{0..40,42,44,45,48\}$ &56& no &no\\
1--8 & full spectrum & $\le 32$ & yes &yes\\
\hline
\end{tabular}
\\[5pt]
Table 37: $n=20$(b), $k=5$, $d=10$, $m_f=8$\\[20pt]
\begin{tabular}{|c|c|c|c|c|}
\hline
$m$ & $\{{\rm minors}\}$ & $\Delta(m)$ & max? & full?\\
\hline
11--20 & as for $20$(b) & -- & -- & no\\
 & $\{0..88,90,92,93,96,99,100,102,$ & & & \\[-4pt]
\raisebox{5pt}{10} & $104,108,112,120,125,128,144\}$ 
 & \raisebox{5pt}{144} & \raisebox{5pt}{yes} & \raisebox{5pt}{no}\\
1--9 & as for 20(b)  & -- & -- & --\\
\hline
\end{tabular}
\\[5pt]
Table 38: $n=20$(c), $k=5$, $d=10$, $m_f=8$\\
\end{center}
\squash
\end{table}

\pagebreak[4]
\begin{table}[ht] 	
\begin{center}
{\bf Order 21(a)--(b)}\\[30pt]
\begin{tabular}{|c|c|c|c|c|}
\hline
$m$ & (min, max) $|$minor$|$ & $\Delta(m)$ & max? & full?\\
\hline
21 & $(29,29)\times 5^{9}$ & $29\times 5^{9}$ & yes &no\\
20 & $(10,30)\times 5^{8}$ & $50\times 5^{8}$ & no &no\\
19 & $(0,35)\times 5^7$ & $833\times 4^6$ & no &no\\
18 & $(0,40)\times 5^6$ & $17\times4^8$ & no & no\\
17 &$(0,45)\times 5^5$& $5 \times 4^8$ & no & no\\ 
16 & $(0,65)\times 5^4$&$2 \times 4^8$&no & no\\
15 & $(0,100)\times 5^3$& $35 \times 3^6$ & no & no\\
14 & $(0,240)\times 5^2$ & $13\times 3^6$ & no & no\\
13 & $(0,416)\times 5^1$ &$5 \times 3^6$& no & no \\
12 & $(0,800)$&$1458$& no & no \\
11 & $(0,320)$ & $320$ & yes & no\\
10 & $(0,144)$ & 144 & yes & no\\ 
1--9 & full spectrum & $\le 56$ & yes &yes\\
\hline
\end{tabular}
\\[5pt]
Table 39: $n=21$(a), $k=5$, $d=10$, $m_f=9$\\[30pt]
\begin{tabular}{|c|c|c|c|c|}	
\hline
$m$ & (min, max) $|$minor$|$ & $\Delta(m)$ & max? & full?\\
\hline
15--21 & as for $21$(a) & -- & -- & no\\
14 & $(0,216)\times 5^2$ & $13\times 3^6$ & no & no\\
13 & $(0,400)\times 5^1$ &$5 \times 3^6$& no & no \\
12 & $(0,800)$&$1458$& no & no \\
11 & $(0,320)$ & $320$ & yes & no\\ 
1--10 & full spectrum & $\le 144$ & yes &yes\\
\hline
\end{tabular}
\\[5pt]
Table 40: $n=21$(b), $k=5$, $d=10$, $m_f=10$\\
\end{center}
\squash
\end{table}

\pagebreak[4]
\begin{table}[ht] 	
\begin{center}
{\bf Order 21(c)--(e)}\\[20pt]
\begin{tabular}{|c|c|c|c|c|}
\hline
$m$ & (min, max) $|$minor$|$ & $\Delta(m)$ & max? & full?\\
\hline
15--21 & as for $21$(a) & -- & -- & no\\
14 & $(0,200)\times 5^2$ & $13\times 3^6$ & no & no\\
13 & $(0,400)\times 5^1$ &$5 \times 3^6$& no & no \\
12 & $(0,800)$&$1458$& no & no \\
11 & $(0,304)$ & $320$ & no & no\\ 
10 & $(0,144)$ & 144 & yes & no\\ 
1--9 & full spectrum & $\le 56$ & yes &yes\\
\hline
\end{tabular}
\\[5pt]
Table 41: $n=21$(c), $k=5$, $d=11$, $m_f=9$\\[20pt] 
\begin{tabular}{|c|c|c|c|c|}
\hline
$m$ & (min, max) $|$minor$|$ & $\Delta(m)$ & max? & full?\\
\hline
15--21 & as for $21$(a) & -- & -- & no\\
14 & $(0,240)\times 5^2$ & $13\times 3^6$ & no & no\\
13 & $(0,416)\times 5^1$ &$5 \times 3^6$& no & no \\
12 & $(0,800)$ &$1458$& no & no \\
11 & $(0,320)$ & $320$ & yes & no\\ 
1--10 & full spectrum & $\le 144$ & yes &yes\\
\hline
\end{tabular}
\\[5pt]
Table 42: $n=21$(d), $k=5$, $d=10$, $m_f=10$\\[20pt]
\begin{tabular}{|c|c|c|c|c|}
\hline
$m$ & (min, max) $|$minor$|$ & $\Delta(m)$ & max? & full?\\
\hline
15--21 & as for $21$(a) & -- & -- & no\\
14 & $(0,212)\times 5^2$ & $13\times 3^6$ & no & no\\
13 & $(0,368)\times 5^1$ &$5 \times 3^6$& no & no \\
12 & $(0,800)$ &$1458$& no & no \\
11 & $(0,288)$ & $320$ & no & no\\ 
10 & $(0,144)$ & 144 & yes & no\\ 
1--9 & full spectrum & $\le 56$ & yes &yes\\
\hline
\end{tabular}
\\[5pt]
Table 43: $n=21$(e), $k=5$, $d=11$, $m_f=9$
\end{center}
\squash
\end{table}

\end{document}